\documentclass[11pt]{amsart}
\usepackage{geometry}                
\usepackage{graphicx}
\usepackage{amssymb}
\usepackage{epstopdf}
\usepackage{amsmath}

\usepackage{amsthm}
\newtheorem{thm}{Theorem}
\newtheorem{mydef}{Definition}
\newtheorem{lem}{Lemma}
\numberwithin{equation}{section}
\numberwithin{lem}{section}
\numberwithin{thm}{section}
\numberwithin{mydef}{section}
\numberwithin{figure}{section}
\title{Hodge-\MakeLowercase{de}Rham theory on fractal graphs and fractals}
\author{Skye Aaron}
\thanks{Research supported by the National Science Foundation through the Research Experiences for Undergraduates at Cornell.}
\author{Zach Conn}
\author{Robert Strichart}
\thanks{ Research supported in part by the National Science Foundation, Grand DMS 0652440.}
\author{Hui Yu}
\thanks{Research supported by the Summer Research Experience for Undergraduates, Department of Mathematics, The Chinese University of Hong Kong.}
\thanks{2010 \underline{Mathematics Subject Classification.} Primary: 28A80. }

\begin{document}
\maketitle

\begin{abstract}
We present a new approach to the theory of k-forms on self-similar fractals. We work out the details for two examples, the standard Sierpinski gasket and 3-dimensional Sierpinski gasket (SG$^3$), but the method is expected to be effective for many PCF fractals, and also infinitely ramified fractals such as the Sierpinski carpet (SC). Our approach is to construct k-forms and deRham differential operators $d$ and $\delta$ for a sequence of graphs approximating the fractal, and then pass to the limit with suitable renormalization, in imitation of Kigami's approach on constructing Laplacians on functions. One of our results is that our Laplacian on 0-forms is equal to Kigami's Laplacian on functions. We give explicit construction of harmonic 1-forms for our examples. We also prove that the measures on line segments provided by 1-forms are not absolutely continuous  with respect to Lebesgue measures.

\end{abstract}

\section{Introduction}
Following the successful development of a differential calculus on certain fractals (\cite{Bar}, \cite{Ki}, \cite{SB}), it would seem natural to try to develop an analogue of the Hodge-deRham calculus of k-forms. In recent years there have been several approaches to this problem (\cite{C}, \cite{CGIS}, \cite{CGIS2}, \cite{CS}, \cite{GI}, \cite{GI2}, \cite{GI3}, \cite{MH}, \cite{IRT}),  concentrating mainly on the case $k=1$. Here we present yet another approach. The fractals we deal with may be regarded as limits of a sequence of graphs, and the calculus of functions as developed by Kigami involves defining the fractal Laplacian as a limit of graph Laplacians, suitably renormalized. Our idea is to regard k-forms on the fractal as limits of k-forms on graphs, and to define the derivative $d$ and $\delta$ as suitably renormalized limits of the corresponding operators on the graphs. In particular, this means understanding the relationships among these objects as we pass from one graph to the next. On the level of the fractal, we will be guided by the intuition that a k-form is an object that can be integrated over k-dimensional subjects, and the operators $d$ and $\delta$ will be defined only for k-forms that are sufficiently smooth. It is by no means obvious how to realize these intuitions in a precise theory, and there are perhaps more than one way to do this.

We will concentrate on one specific fractal, the Sierpinski gasket (SG), which has become the `poster child' for Kigami's class of postcritically finite (PCF) self-similar fractals. SG is the limit of graphs $G_0, G_1, G_2,\cdots$, shown in Figure (1.1).

\begin{figure}
\scalebox{0.5}{\includegraphics{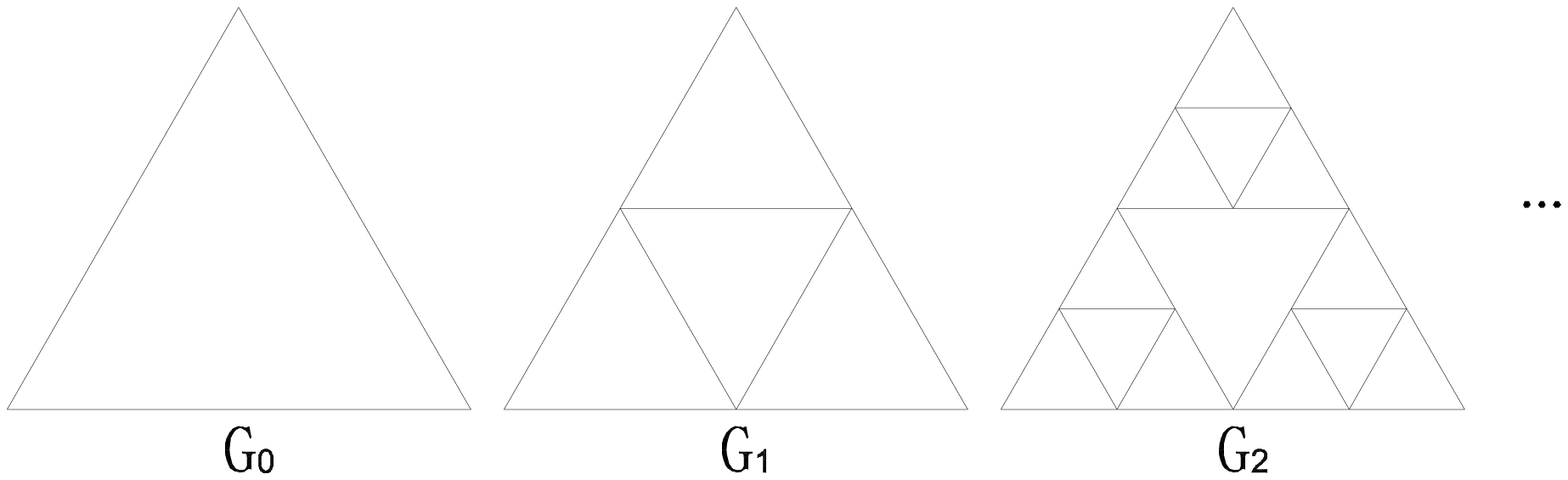}}
\caption{}
\end{figure}
From the topological point of view, SG has 1-dimensional homology generated by infinitely many independent cycles that are visible as `upside-down' triangles in the graphs $G_m$. It also contains infinitely many straight line segments that are also visible as edges in the graphs. We will therefor define 0-forms on SG as functions, 1-forms as measures on the line segments, and 2-forms as measures on SG. The `obvious' definition of $d_0$ from 0-forms to 1-forms is \begin{equation}
\int_{L}d_0f_0=f(b)-f(a)
\end{equation}  
if $a$ and $b$ are the endpoints of L. We will define the dual operator $\delta_1$ from 1-forms to 0-forms, which leads to a Laplacian $\Delta_0=-\delta_1 d_0$ on 0-forms. We will show that this Laplacian exactly coincides with Kigami's Laplacian on functions. Note that (1.1) requires $f_0$ to be continuous, but also the restriction of $f_0$ to any line segment must be of bounded variation so that the restriction of $d_0f_0$ to the line segment is in fact a measure. (Finite additivity is obvious from (1.1), but countable additivity is the issue.) We will show that the class of functions in the domain of Laplacian satisfy this condition. However, it turns out that the measures $d_0f_0$ for this class of functions are singular with respect to Lebesgue measure on the line segments.  

There is also an `obvious' definition of $d_1$ from 1-forms to 2-forms, namely,
\begin{equation}
d_1f_1(C)=\int_{\partial C}f_1
\end{equation} 
where C denotes any cell (the interior of a small triangle in one of the $G_m$ graphs). Here $\partial C$ consists of the oriented boundary line segment of the triangle, so the right side of (1.2) is just the sum of $f_1$ on three edges. (To specify a measure it suffices to give its values on cells of all levels.) It is clear from (1.1) that $d_1d_0f_0=0$. We define harmonic 1-forms as solutions to the two equations $d_1f_1=0$ and $\delta_1f_1=0$. We may then consider the cohomology/ homology pairing
\begin{equation}
\int_{\gamma}h
\end{equation} 
between harmonic 1-forms $h$, and homology cycles $\gamma$. We will explicitly construct a basis of harmonic 1-forms that give zero in (1.3) for all but a single homology generator $\gamma$. Since this is an infinite basis, there is an issue in how to deal with the closure of the span. We are unable to give a completely satisfactory resolution of the problem. 

The harmonic 1-forms gives one of the three pieces of an expected Hodge decomposition of 1-forms. Another piece is the image under $d_0$ of the 0-forms. In this piece the Laplacian is given by $\Delta_1=-d_0\delta_1$, since $d_1$ annihilates this piece of 1-forms. It is straightforward to observe that the spectrum of this portion of $-\Delta_1$ coincides with the spectrum of $-\Delta_0$, except for the zero eigenvalue which corresponds to the constant eigenfunction that is annihilated by $d_0$. In other words, if $f_0$ is an eigenfunction of $-\Delta_0$, so $-\Delta_0f_0=\lambda f_0$, with $\lambda\neq 0$, then $d_0f_0$ is an eigenfunction of $-\Delta_1$, so $-\Delta_1 d_0f_0=\lambda d_0f_0$. For the third piece of the Hodge decomposition we need to define $\delta_2$  as the dual of $d_1$. This is a somewhat ambiguous problem, as it depends on what class of 2-forms (measures) we would like to be able to apply $\delta_2$ to, and what class of 1-forms we would like to get as the image. Roughly speaking, $\delta_2f_2$ on an edge should be some sort of a `trace' of the measure $f_2$ on the edge, i.e., some renormalized limit of the measure of an $\epsilon$-neighborhood of the edge. We give two distinct realizations of this intuition, using different renormalizations. One allows 2-forms $f_2$ that are absolutely continuous with respect to the standard self-similar measure on SG, with continuous Radon-Nikodym derivative. The other allows 2-forms $f_2$ that are absolutely continuous with respect to the Kusuoka measure. (These measures are mutually singular, as first shown by Kusuoka.) In both cases the Laplacian $-\Delta_2$ on 2-forms and $-\Delta_1=\delta_2d_1$ on the $\delta_2$ portion of 1-forms are just multiples of the identity. This triviality is already seen on the graph level, and can be explained by the observation that cells only intersect at vertices, not edges. Every edge belongs to a unique cell, and boundaries of different cells are disjoint. This observation is valid on all PCF fractals.

In other words, to get a nontrivial calculus of k-forms for $k>1$ will require going outside the realm of PCF fractals. This perhaps explains why previous works have concentrated on the $k=1$ case. We believe that our method can be implemented for fractals such as the Sierpinski carpet (SC), but this remains to be seen.

The organization of this paper is as follows. In Section 2 we summarize the theory of k-forms on graphs. This material is mainly well-known, but we give the details for the convenience of the reader, and to set the notation for the remainder of the paper. In Section 3 we outline the steps involved in passing from one approximating graph to the next, on the route from fractal graphs to fractals. The actual passage to the limit would seem to require a careful analysis for each example.

In Section 4 we study the example SG. We present all the results described above, except for the proof of the singularity of the measure $d_0f_0$ on the line segments, which is presented in Section 5. This result is likely to be of independent interest.

In Section 6 we examine briefly another example, the 3-dimensional analog of the Sierpinski gasket (SG$^3$). We concentrated on what we believe to be the most interesting result, namely the explicit computation of harmonic 1-forms. We believe that similar results should be valid for many other PCF fractals.  

\section{Graphs}
Let $G$ be a finite connected graph with $E_0$ and $E_1$ the collections of its vertices and edges, respectively. For $e_0\in E_0$ and $e_1\in E_1$, $e_0\subset e_1$ means $e_0$ is one of the vertices of $e_1$. For $k\ge 2$, if we take $E_k$ to be the collection of complete $k+1$ subgraphs of $G$, then we obtain a hierarchy $E_0, E_1, E_2,\cdots, E_n$ of subgraphs.  For $e_k\in E_k$ and $e_{k+1}\in E_{k+1}$, $e_k\subset e_{k+1}$ means the vertices of $e_k$ are all vertices of $e_{k+1}$.

More generally, we may choose a set of model connected graphs $\Gamma_0,\Gamma_1,\Gamma_2,\cdots\Gamma_n$ with $\Gamma_0$ being a single vertex and $\Gamma_1$ being two vertices connected by an edge, and define $E_k$ to be a collection of subgraphs of $G$ isomorphic to $\Gamma_k$ (but not necessarily all of them). Recall that a subgraph consists of a certain subset of the vertices with all the edges in $G$ connecting them, that is, we are not allowed to throw away edges to obtain an isomorphism. To have a nontrivial theory we require  that $\Gamma_k$ is isomorphic to a subgraph of $\Gamma_{k+1}$, and that each vertex in $G$ lies in some $e_{k}$ for every $k$. We will also assume that for every $e_k\in E_k$ there exist $e_{k-1}\in E_{k-1}$ (unless $k=0$) and $e_{k+1}\in E_{k+1}$ (unless $k=n$) such that $e_{k-1}\in e_{k}\in e_{k+1}$. 

For example, if $G$ is an approximate Sierpinski gasket graph, $\Gamma_2$ is  the complete 3-graph. If $G$ is an approximate graph to the n-dimensional Sierpinski gasket,  $\Gamma_{k}$ is the complete (k+1)-graph for all $k\le n$. For a different type of example, let $G$ be a subgraph of the square lattice and take $\Gamma_2$ to be the 4-loop.
\begin{mydef} A parity function $sgn(e_k, e_{k+1})$ is a function, defined for all $e_k\in E_k$ and $e_{k+1}\in E_{k+1}$ $(k\le n-1)$, taking values in $\{-1,0,1\}$ and satisfying:
\begin{equation}sgn(e_k, e_{k+1})\neq 0 \Leftrightarrow e_k\subset e_{k+1}\end{equation}and
\begin{equation}\sum_{e_k\in E_k}sgn(e_{k-1},e_k)sgn(e_k,e_{k+1})=0\end{equation}
for all $e_{k-1}\subset e_{k+1}$.

\end{mydef}

For example, assign orientations to all subgraphs in all $E_k$'s and define $sgn(e_k,e_{k+1})=~1$ if the orientations of $e_k$ and $e_{k+1}$ are consistent, and $sgn(e_k,e_{k+1})=-1$ if their orientations are inconsistent. Then the sum in (2.2) contains exactly 2 nonzero terms with values $1$ and $-1$. Therefore we have a parity function. 

\begin{mydef}A set of weights is an assignment of positive values $\mu_k (e_k)$ to each $e_k\in E_k$.

The space $D_k$ of k-forms is the set of functions $f_k:E_k\to \mathbb{C}$ with an inner product structure\begin{equation}<f_k,g_k>_k=\sum_{e_k\in E_k}\mu_k(e_k)f_k(e_k)\overline{g_k(e_k)}.\end{equation}
\end{mydef}

\begin{mydef}The deRham derivative $d_k:D_k\to D_{k+1}$ is defined by
\begin{equation}d_k f_k(e_{k+1})=\sum_{e_k\in E_k}sgn(e_k,e_{k+1})f_k(e_k).\end{equation}
for $k\le n-1$ and $d_{n}f_n\equiv 0$.

The dual deRham derivative $\delta_k: D_k\to D_{k-1}$ is defined by duality
\begin{equation}<\delta_k f_k,g_{k-1}>_{k-1}=<f_k,d_{k-1}g_{k-1}>_k\end{equation} for $k\ge 1$ and $\delta_{0}f_0\equiv 0$.
\end{mydef}

An explicit formula for $\delta_k$ follows from direct computation
\begin{equation}\delta_k f_k(e_{k-1})=\sum_{e_k}\frac{\mu_k(e_k)}{\mu_{k-1}(e_{k-1})}sgn(e_{k-1},e_k)f_k(e_k).\end{equation}

We have the following theorem:
\begin{thm}$d_k d_{k-1}=0$ and $\delta_k \delta_{k+1}=0$.\end{thm}
\begin{proof}\[d_k d_{k-1} f_{k-1}(e_{k+1})=\sum_{e_k}\sum_{e_{k-1}}sgn(e_k,e_{k+1})sgn(e_{k-1},e_{k})f_{k-1}(e_{k-1})\]so (2.2) implies $d_{k}d_{k-1}=0$. 

$\delta_k \delta_{k+1}=0$ follows from duality.\end{proof}

This theorem shows that the deRham complex
\begin{equation}D_0\overset{d_0}\longrightarrow D_1\overset{d_1}\longrightarrow D_2\cdots\overset{d_{n-1}}\longrightarrow D_n\overset{d_n}\longrightarrow 0\end{equation}
is an exact sequence, and so is the dual deRham complex
\begin{equation}D_n\overset{\delta_n}\longrightarrow D_{n-1}\overset{\delta_{n-1}}\longrightarrow D_{n-2}\cdots\overset{\delta_1}\longrightarrow D_{0}\overset{\delta_{0}}\longrightarrow 0.\end{equation}

Thus it makes sense to define the deRham cohomology spaces
\begin{equation}H_k=ker (d_k)/range( d_{k-1}).\end{equation}

\begin{mydef}The energy $\mathcal{E}_k$ is the symmetric bilinear form on $D_k$ given by
\begin{equation}
\mathcal{E}_k(f_k,g_k)=<d_kf_k,d_kg_k>_{k+1}+<\delta_kf_k,\delta_kg_k>_{k-1}\end{equation}(only one term if $k=0$ or $k=n$). 

The Laplacian $\Delta_k$ is the symmetric operator on $D_k$ given by
\begin{equation}-\Delta_k=\delta_{k+1}d_k+d_{k-1}\delta_{k}.\end{equation}

$f_k\in D_k$ is a harmonic k-form if \begin{equation}-\Delta_k f_k=0\end{equation}and the space of harmonic k-forms is denoted by $\mathcal{H}_k$. 
\end{mydef}

\begin{thm}$-\Delta_k$ is the operator associated to $\mathcal{E}_k$, namely,
\begin{equation}<-\Delta_k f_k,g_k>_k=\mathcal{E}_k(f_k,g_k)\end{equation} and this characterizes $-\Delta_k f_k$. 

In paricular, $f_k$ is harmonic if and only if 
\begin{equation}\mathcal{E}_k(f_k,f_k)=0.\end{equation} 
Or equivalently, 
\begin{equation}d_k f_k=0,\ and\  \delta_k f_k=0\end{equation}(only one condition if $k=0$ or $k=n$).
\end{thm}
\begin{proof}
(2.13) follows from the definitions and duality, and this gives $-\Delta_k f_k(\bar{e_k})$ by using $g_k(\bar{e_k})=\delta(e_k,\bar{e_k})$.

(2.13) also implies that harmonic forms satisfy (2.14). Conversely, (2.14) implies $\mathcal{E}_k(f_k,g_k)=0$ for all $g_k\in D_k$ by polarization and hence $\Delta_k f_k=0$.

(2.15) implies $\Delta_k f_k=0$ trivially. Conversely, (2.13) implies \[<d_kf_k,d_kg_k>_{k+1}+<\delta_kf_k,\delta_kg_k>_{k-1}=0\] and hence (2.15) follows from positive definiteness of inner products.
\end{proof}

For example we can compute $-\Delta_0=\delta_1 d_0$ as:
\begin{align}\delta_1 d_0 f_0(e_0)&=\sum_{e_1\in E_1}\frac{\mu_1(e_1)}{\mu_0(e_0)}sgn(e_0,e_1)d_0f_0(e_1)\\ \notag&=\sum_{e_1\in E_1}\frac{\mu_1(e_1)}{\mu_0(e_0)}sgn(e_0,e_1)\sum_{e_0'\in E_0}sgn(e_0',e_1)f_0(e_0')\\ \notag&=\sum_{[e_0,e_0']\in E_1}\frac{\mu_1(e_1)}{\mu_0(e_0)}(f_0(e_0)-f_0(e_0')).\end{align}

The first equality follows from (2.6) and the last follows from the fact that every edge $e_1$ has two vertices $e_0$ and $e_0'$ with $sgn(e_0,e_1)=-sgn(e_0',e_1)$. Note that (2.16) is the standard graph Laplacian with weights $\mu_0$ and $\mu_1$ on vertices and edges [CdV].

\begin{thm}\textup{\textbf{Hodge Decomposition}} For each $k$, we have
\begin{equation}D_k=d_{k-1}D_{k-1}\oplus\delta_{k+1} D_{k+1}\oplus\mathcal{H}_k\end{equation} as an orthogonal direct sum. Thus $H_k$ is isomorphic to $\mathcal{H}_k$.\end{thm}

\begin{proof}By Theorem 2.1, \[<d_{k-1}f_{k-1},\delta_{k+1}g_{k+1}>_k=<d_kd_{k-1}f_{k-1},g_{k+1}>_{k+1}=0,\] thus the first two terms in the decomposition are orthogonal.

If $f_k\in\mathcal{H}_k$ then $d_k f_k=0$ and $\delta_{k}f_k=0$, so \[<f_k,\delta_{k+1}g_{k+1}>_k=<d_k f_k,g_{k+1}>_{k+1}=0\] and \[<f_k,d_{k-1}g_{k-1}>_{k}=<\delta_{k}f_k,g_{k-1}>_{k-1}=0.\] Thus $\mathcal{H}_k\subset (d_{k-1}D_{k-1}\oplus\delta_{k+1} D_{k+1})^{\perp}$. 

Conversely, $<f_k, d_{k-1}g_{k-1}>_k=0$ and $<f_k,\delta_{k+1}g_{k+1}>_{k}=0$ for all $g_{k-1}\in D_{k-1}$ and $g_{k+1}\in D_{k+1}$ imply $d_k f_k=0$ and $\delta_k f_k=0$, and therefore the orthogonal decomposition is proved.

By (2.17) we have $ker(d_k)\supset d_{k-1}D_{k-1}\oplus\mathcal{H}_k$. On the other hand, with \[<d_k \delta_{k+1}g_{k+1},g_{k+1}>_{k+1}=<\delta_{k+1}g_{k+1},\delta_{k+1}g_{k+1}>_{k}\] we have $d_k\delta_{k+1}g_{k+1}\neq 0$ if $\delta_{k+1}g_{k+1}\neq 0$. Thus $ker(d_k)= d_{k-1}D_{k-1}\oplus\mathcal{H}_k.$ This implies $\mathcal{H}_k$ is the orthogonal complement of $range(d_{k-1})$ in $ker(d_{k})$.
\end{proof}

\begin{mydef} A k-chain $C_k$ is a formal sum $C_k=\sum_{e_k\in E_k}a_k e_k$ with $a_k\in\mathbb{C}$. The collection of k-chains is denoted by $\mathcal{C}_k$.

The boundary operator $\partial_{k}:\mathcal{C}_k\to\mathcal{C}_{k-1}$ is given by 
\begin{equation}\partial_k e_k=\sum_{e_{k-1}\in E_{k-1}}sgn(e_{k-1},e_k)e_{k-1}\end{equation}and
\begin{equation}\partial_k(\sum_{e_k\in E_k}a_k e_k)=\sum_{e_k\in E_k}a_k\partial_k e_k.\end{equation}

A k-chain $C$ is a k-cylce if $\partial_k C_k=0$.

Integration of k-forms along k-chains is defined by 
\begin{equation}\int_{C_k}f_k=\sum_{e_k\in E_k}a_k f_k(e_k).\end{equation} 
\end{mydef}

The above definition gives a duality between $D_k$ and $\mathcal{C}_k$ and the following theorem is an immediate consequence of the definitions.

\begin{thm}\textup{\textbf{Stokes' Theorem}}\begin{equation}\int_{C_k}df_{k-1}=\int_{\partial_k C_k}f_{k-1}.\end{equation}

In particular \begin{equation}\int_{C_k}df_{k-1}=0\end{equation} if $C_k$ is a k-cycle.

\end{thm}

Also note that (2.2) implies $\partial_{k-1}\partial_{k}=0$ and hence the boundary complex
\begin{equation}\mathcal{C}_n\overset{\partial_n}\longrightarrow\mathcal{C}_{n-1}\overset{\partial_{n-1}}\longrightarrow\mathcal{C}_{n-2}\cdots\overset{\partial_1}\longrightarrow\mathcal{C}_0\overset{\partial_0}\longrightarrow 0\end{equation} is exact.

A special case of 1-chains concerns paths $\gamma=\sum_j(\pm)e_1^{j}$ of consecutive edges with the signs chosen so that the orientations are consistent. Then (2.21) says
\begin{equation}\int_{\gamma}df_0=f_0(q)-f_0(p)\end{equation} where $p$ and $q$ are endpoints of the path, which gives us a form of the fundamental theorem of calculus. Similarly (2.22) says 
\begin{equation}\int_{\gamma}df_0=0\end{equation}if $\gamma$ is a closed path.

It is sometimes important to compute the dimensions of spaces in the Hodge decomposition. Note that part of the proof of Theorem 2.3 gives $d_k D_k=d_k \delta_{k+1}D_{k+1}$ so
\begin{equation}dim(d_kD_k)=dim(\delta_{k+1}D_{k+1}).\end{equation}It follows that
\begin{equation}dim(D_k)=dim(d_{k-1}D_{k-1})+dim(d_{k}D_k)+dim(\mathcal{H}_k).\end{equation}

In particular, since constants are the only harmonic 0-forms we have $dim(\mathcal{H}_0)=1$ and 
\begin{equation}dim(d_0 D_0)=dim (D_0)-1.\end{equation}

Now we look into the spectra of Laplacians $-\Delta_k$. $\mathcal{H}_k$ is the 0-eigenspace of $-\Delta_k$. The remaining eigenspaces can be split among the other two terms of the Hodge decomposition, which we call d-spectrum($-\Delta_k$) (the eigenfunctions and associated eigenvalues in $d_{k-1}D_{k-1}$) and $\delta$-spectrum($-\Delta_k$)(the eigenfunctions and associated eigenvalues in $\delta_{k+1}D_{k+1}$). The two have the following relation:

\begin{thm}For $k\ge 1$, d-spectrum($-\Delta_k$) and $\delta$-spectrum($-\Delta_{k-1}$) contain the same eigenvalues (counting multiplicities), and the corresponding eigenfunctions $f_k$ and $f_{k-1}$ are related by $f_k=d_{k-1}f_{k-1}$ and $f_{k-1}=\delta_{k}f_k$.\end{thm}

\begin{proof}If $f_{k-1}$ is in $\delta$-spectrum($-\Delta_{k-1}$) with eigenvalue $\lambda$, then $\delta_k d_{k-1}f_{k-1}=\lambda f_{k-1}$ since $d_{k-2}\delta_{k-1}f_{k-1}=0$. We then have 
\[-\Delta_k(d_{k-1}f_{k-1})=d_{k-1}\delta_k d_{k-1}f_{k-1}=\lambda d_{k-1}f_{k-1}.\]
So $d_{k-1}f_{k-1}$ is a $\lambda$-eigenfunction in d-spectrun($-\Delta_{k}$).

Similarly if $f_{k}$ is a $\lambda$-eigenfunction in d-spectrum($-\Delta_k$), then $\delta_k f_k$ is a $\lambda$-eigenfunction in $\delta$-spectrum($-\Delta_{k-1}$).
\end{proof}

\section{From Graphs to Fractals}
Suppose $\mathcal{K}$ is a fractal that can be realized, in some way, as the limit of graphs $\{G_m\}_{m=0}^{\infty}$. We want the theory of k-forms on $G_m$'s to give a theory of k-forms on $\mathcal{K}$ in the limit after necessary renormalization. In this section we discuss some general strategy for this before we work on specific examples in later sections.

Kigami \cite{Ki} has introduced a family of fractals, post-critically finite (PCF) self-similar fractals, that can be realized as attractors of some iterated function systems (IFS), say, $\{F_j\}_{j=1}^N$ on Euclidean spaces (The ambient space plays no role in the theory and is introduced to simplify the discussion). We assume that $F_j$'s are contractive similarities, and that there exists a finite set $V_0$, the boundary of $\mathcal{K}$, consisting of fixed points $q_j$ of some of the $F_j$ such that  
\begin{equation}F_j\mathcal{K}\cap F_k\mathcal{K}\subset F_j V_0\cap F_k V_0\end{equation}for all $j\neq k$. In particular, the \textit{cells at level 1}, $\{F_j\mathcal{K}\}_{j=1}^N$, can only intersect at finite many points. But on the other hand, we assume that there are enough intersections to keep $\mathcal{K}$ connected, and note that (3.1) implies that $\mathcal{K}$ becomes disconnected if the finite set $V_0$ is removed (a property sometimes called \textit{finite ramification}).

More generally, for any \textit{word} $\omega=(\omega_1,\omega_2,\omega_3,\cdots,\omega_m)$ of length $|\omega|=m$, define \[F_{\omega}=F_{\omega_1}\circ F_{\omega_2}\circ F_{\omega_3}\circ\cdots \circ F_{\omega_m}\]and 
\[\mathcal{K}_{\omega}=F_{\omega}\mathcal{K},\] a cell of level $m$. Again we assume $K_{\omega}\cap\mathcal{K}_{\omega'}\subset F_{\omega}V_0\cap F_{\omega'}V_0$ for any $\omega\neq\omega'$ and $|\omega|=|\omega'|$(This does not necessarily follow from (3.1)).
 
Now for such a PCF fractal we have a natural sequence of graph approximations. Let $G_0$ be the complete graph with vertices in $V_0$. Let $V_1=\{F_j q_k\}$, the images of the boundary points under the mappings in the IFS. $G_1$ is then the graph with points in $V_1$ as vertices and images of edges of $G_0$ under the IFS as its edges. In other words, $G_1$ consists of $N$ copies of $G_0$ with certain identified vertices. Note that $V_0\subset V_1$ but the edges in $G_0$ are not necessarily edges in $G_1$. 

Iterating this procedure we obtain a sequence of graphs $G_0, G_1,G_2\cdots$ with sets  of vertices 
\begin{equation}V_{0}\subset V_1\subset V_2\subset \cdots.\end{equation} Then the subgraphs $F_{\omega}G_0$  of $G_{m}$  $(|\omega|=m)$ are cells of level m. Note that vertices in $V_m$ are all points in the actual fractal and $V_{*}=\cup V_m$ is dense in $\mathcal{K}$ in the natural topology.

Now we choose model graphs $\Gamma_0,\Gamma_1,\Gamma_2,\cdots\Gamma_N$ and define $E_k^0$ to be a collection of subgraphs of $G_0$ isomorphic to $\Gamma_k$. For each $m$, let $E_k^m$ be the union of images of $F_{\omega}(E_k^0)$ for $|\omega|=m$. Each graph in $E_k^m$ then belongs to a single m-cell in $G_m$ for $k\ge 1$ but vertices in $E_0^m$ might belong to several cells due to identification.

Once a function $sgn(e_k,e_{k+1})$ is defined on $G_0$ satisfying (2.1) and (2.2), it can be transported to $G_m$'s in a natural way, and (2.1) and (2.2) still hold. Similarly given a set of weights $\mu_k$ on $E_k^0$ and a collection of positive numbers $\{b_k^j\}_{j=1}^N$, we can define a set of weights on $E_k^m$ via:
\begin{equation}\mu_k^m (e_k^m)=b_k^{\omega}\mu_k^0(e_k^0)\end{equation} if $e_k^m=F_{\omega}e_k^0$ with $|\omega|=m$ and $k\ge 1$, and 
\begin{equation}\mu_0^m(e_0^m)=\sum_{F_{\omega}e_0^0=e_0^m}b_k^{\omega}\mu_0^0(e^0_0)\end{equation} if $k=0$, where $b_k^{\omega}=\prod_{j=1}^m b_k^{\omega_j}$.

In this way a sequence of Laplacians and deRham complexes can be defined as in Section 1. The problem is to relate the structures when $m$ varies and to pass to a limit when $m\to\infty$. We deal with this on a case-by-case basis.

Kigami \cite{Ki} defined a notion of regular harmonic structure on PCF fractals, which is closely related to our setup and especially to the weights $\mu_1$. Recall that energy on $G_0$ is given by
\begin{equation}\mathcal{E}_0^0(f_0^0,g_0^0)=\sum_{e_1^0\in E_1^0}\mu_1(e_1^0)(f_0^0(y)-f_0^0(x))(g_0^0(y)-g_0^0(x)),\end{equation}where $e_1^0=[x,y]$, for pairs of functions $f_0^0, g_0^0$ on $V_0=E_0^0$. With (3.3) this extends to energies $\mathcal{E}_0^m$ on $G_m$ via
\begin{equation}\mathcal{E}_0^m(f_0^m,g_0^m)=\sum_{|\omega|=m}\sum_{e_1^0\in E_1^0}b_1^{\omega}\mu_1(e_1^0)(f_0^m(F_{\omega}y)-f_0^m(F_{\omega}x))(g_0^m(F_{\omega}y)-g_0^m(F_{\omega}x)).\end{equation}

Note that a function $f_0^{m-1}$ on $V_{m-1}=E_0^{m-1}$ can be extended to $f_0^m$ on $V_m=E_0^m$ in many ways. The extension $\widetilde{f}_0^m$ that minimizes $\mathcal{E}_0^m(f_0^m)=\mathcal{E}_0^m(f_0^m,f_0^m)$ is called the harmonic extension of $f_0^{m-1}$. To have a harmonic structure we require $\mathcal{E}_0^m(\widetilde{f}_0^m)=\mathcal{E}_0^{m-1}(\widetilde{f}_0^{m-1})$ and the structure is regular if $0<b_1^j<1$ for all $j$. There is a rather large literature on harmonic structures (see for example [Sa], [CS], [P]).

Fix a regular harmonic structure. We obtain the energy $\mathcal{E}_0$ on $\mathcal{K}$ by 
\begin{equation}\mathcal{E}_0(f)=\lim_{m\to +\infty}\mathcal{E}_0^m(f_0|_{V_m}).\end{equation} The limit always exists because $\mathcal{E}_0^m(f_0|_{V_m})$ is an increasing sequence. Define $dom(\mathcal{E}_0)$ to be the collection of functions such that $\mathcal{E}_0(f)<\infty$. Such functions are always continuous and hence the restriction to the dense set $V_{*}$ determines the function.
Define the quadratic form on $dom(\mathcal{E}_0)$ by 
\begin{equation}\mathcal{E}_0(f,g)=\lim_{m\to\infty}\mathcal{E}_0^{m}(f|_{V_m},g|_{V_m})\end{equation}Only constant functions have zero energy and the space $dom(\mathcal{E}_0)$ modulo constants is a Hilbert space with the energy inner product. See \cite{Ki} for details.

We may also use the weights $\mu_0$ to define a measure on $\mathcal{K}$ as the weak limit of
\begin{equation}\sum_{e_0^m\in E_0^m}\mu^m_0(e_0^m)\delta_{e_0^m}.\end{equation}To obtain a probability measure we need
\begin{equation}\sum_{e_0^0\in E_0^0}\mu_0^0(e_0^0)=1\end{equation}and
\begin{equation}\sum_j b_0^j=1.\end{equation}
We denote this measure also by $\mu_0$ and note that $\mu_0(F_{\omega}\mathcal{K})=b_0^{\omega}$ so the measure does not depend on the initial distribution of weights on $V_0$. Now since the discrete Laplacians $-\Delta_m$ on $G_m$ is given by (see (2.16))
\begin{equation}-\Delta_0^m f_0^m(x)=\sum_{[x,y]\in E_1^m}\frac{\mu_1^m([x,y])}{\mu_0^m(x)}(f_0^m(x)-f_0^m(y))\end{equation} we might define a Laplacian $-\Delta_0$ on $\mathcal{K}$ by
\begin{equation}-\Delta_0f_0(x)=\lim_{m\to\infty}-\Delta_0^m(f_0|_{V_m}),\end{equation}and define $dom(\Delta_0)$ to be the space of functions such that the above limit exists uniformly for all points in $V\backslash V_0$.

Or equivalently, we can define $-\Delta_0$ by the weak formulation 
\begin{equation}\mathcal{E}_0(f_0.g_0)=\int(-\Delta_0 f_0)g_0d\mu_0\end{equation}
for all $g_0\in dom(\mathcal{E}_0)$ with $g|_{V_0}=0$. Thus we see that the Laplacians defined by Kigami is equivalent to the deRham Laplacian on 0-forms. In particular the harmonic 0-forms are constants.

Higher order forms will be discussed in later sections. However, the fact that $e_k^m$ for $k\ge 1$ lies in a single m-cell implies that $-\Delta_k^m$ is a multiple of the identity, and so any limit, $-\Delta_k$, we might obtain is again a multiple of the identity. Thus in PCF cases, only the theory of 0-forms and 1-forms is nontrivial.

\section{Sierpinski Gasket}
The Sierpinski gasket can be realized in the plane as the attractor of the IFS \[F_j(x)=\frac{1}{2}(x-q_j)+q_j,\] where $\{q_j\}$ are vertices of a regular triangle. That is, $V_0=\{q_1,q_2,q_3\}$ is the set of vertices in the initial graph $G_0$, and subsequent graphs $G_m$ are obtained by applying the IFS to $G_0$ as described in Section 3. Let $E_0^0$ be the collection of vertices in $G_0$, $E_1^0$ the collection of edges in $G_0$, oriented counterclockwise, and $E_2^0$ the triangle, and apply the IFS iteratively to obtain $E_k^m$ $(m=1,2,\cdots)$. Define the parity function $sgn$ using this orientation as described in Section 2. Note that there are complete 3-graphs in $G_m$ that are not in $E_2^m$, namely, the upside-down triangles.

We use the most symmetric weights for $G_0$, giving each point in $V_0$ the weight $1/3$, each edge in $G_0$ the weight $1$ and the triangle the weight $1$. Then for $m\ge 1$,
\begin{equation}\begin{cases}
\mu_0(e_0^m)&=\begin{cases}\frac{1}{3^m}\ (e_0^m\in V_0)\\ \frac{2}{3^m}\ (e_0^m\in E_0^m\backslash V_0)\end{cases}\\
\mu_1(e_1^m)&=(\frac{5}{3})^m\\
\mu_2(e_2^m)&=\frac{1}{3^m}.\end{cases}\end{equation}
In terms of (3.3) and (3.4), we have chosen $b_0^j=\frac{1}{3}, b_1^j=\frac{5}{3},b_2^j=\frac{1}{3}$. Such
renormalization for $\mu_0$ and $\mu_2$ makes them measures on $E_0^m$ and m-cells, respectively. They both converge to the standard probability measure on the Sierpinski gasket. $\mu_1$ is renormalized to give the energy on the gasket. 

Now suppose $f_0$ is a continuous function (a 0-form) on SG; restriction of $f_0$ to $E_0^m$ gives a 0-form on $G_m$. The inner product\begin{equation}<f_0,g_0>_0=\int_{SG}f_0g_0d\mu\end{equation}is the same as the limit of inner product on graphs.

Next we define 1-forms on SG in such a way that the restriction to $E_1^m$ gives a 1-form on $G_m$ (Note that edges in $E_1^m$ are oriented curves in SG). To do this, consider the vector space of $\sum c_j\gamma_j$, where $c_j\in \mathbb{C}$ and $\gamma_j\in\cup_m E_1^m$. We define 1-forms to be elements in the dual space subject to the consistency condition:
\begin{equation}f_1(e_1^m)=\sum_{e_1^k\in e_1^m}f_1(e_1^k)\end{equation} for all $k\ge m$.
Equivalently, $f_1$ defines a signed measure on each edge $e_1^m\in E_1^m$. In particular, if $f_0$ is a 0-form whose restriction to each edge is of bounded variation, then
\begin{equation}d_0f_0([p,q])=f_0(q)-f_0(p)\end{equation} gives a 1-form.  $f_0$ being continuous means the measure is continuous when restricted to each edge. This is too large a class of functions to yield an interesting theory. 

Recall that in Section 3 we defined a space $dom(\mathcal{E}_0)$ to be the collection of functions such that $\mathcal{E}_0(f)<\infty$, where $\mathcal{E}_0$ is defined (see (3.7)), with our choice of weights, by  
\begin{equation}\mathcal{E}_0(f_0)=\lim_{m\to\infty}(\frac{5}{3})^m\sum_{e_1^m\in E_1^m}|d_0f_0(e_1^m)|^2.\end{equation}
Or, equivalently, in terms of the inner product for 1-foms,
\begin{equation}<f_1,f_1>_{1}=\sup_{m\to\infty}(\frac{5}{3})^m\sum_{e_1^m\in E_1^m}|f_1(e_1^m)|^2,\end{equation}we can define 
\begin{equation}\mathcal{E}_0(f_0)=<d_0f_0,d_0f_0>_1.\end{equation}
Note that the right-hand side of (4.5) is always nondecreasing thus we might replace $sup$ with $lim$. However, this is not true for general 1-forms. Also, technical problems concerning convergence arise if we would like to interpret (4.6) as a quadratic form associated with an inner product $<f_1,g_1>_1$. But $\mathcal{E}_0(f_0)=\mathcal{E}_0(f_0,f_0)$ is associated with the quadratic form $\mathcal{E}_0(f_0,g_0)$ since $f_0\in dom(\mathcal{E}_0)$ satisfies the following estimate of H$\ddot{o}$lder type:
\begin{equation}|d_0f_0(e_1^m)|\le(\frac{3}{5})^{m/2}\mathcal{E}_0(f_0).\end{equation} Later we show $f_0\in dom(\Delta)$ implies that the restriction of $f_0$ to edges is of bounded variation and hence $d_0f_0$ gives a finite measure.

Now suppose $f_1$ is a 1-form on SG and write $f_1^m$ to be the restriction to $E_1^m$. For $e_0^m\in E_0^m\backslash V_0$ one has 
\begin{equation}\delta_1^m f_1^m(e_0^m)=\frac{3}{2}5^m\sum_{e_1^m\in E_1^m}sgn(e_0^m,e_1^m)f_1^m(e_1^m),\end{equation}
and there are exactly four nonzero terms in the sum, two with $sgn=+1$ and two with $sgn=-1$. For any nonboundary point in the dense set $V_{*}=\cup_{m}E_0^m$, the expression (4.9) makes sense for $m$ sufficiently large, so $\delta_1f_1$ may be defined as the limit of (4.9) as $m\to\infty$. We would like, for a certain class of 1-forms, to have this limit exists uniformly and the limit to be continuous so as to be extended to a continuous function on SG. For 1-forms of the form $d_0f_0$ where $f_0\in dom(\mathcal{E}_0)$, in particular, we have
\begin{equation}\delta_1^m(d_0f_0)^m(x)=\frac{3}{2}5^m\sum_{[x,y]\in E_1^m}(f_0(x)-f_0(y))=-\Delta_0^m f^m_0(x)\end{equation}Here $\Delta_0^m$ is both the graph Laplacian on $G_m$ and the m-level approximation for the Kigami Laplacian on SG. Thus $f_0\in dom(\Delta_0)$ if and only if (4.10) converges uniformly  to a continuous function on SG. So $d_0dom(\mathcal{E}_0)$ is a space on which $\delta_1$ is well-defined and $\delta_1 d_0=-\Delta_0$. Since the Laplacian for 0-forms in $dom(\Delta)$ agrees with the Kigami Laplacian, they have the same eigenvalues and eigenfunctions. Moreover, if $-\Delta_0 f_0=\lambda f_0$ for some $\lambda\neq 0$, then \[-\Delta_1 d_0f_0= -d_0\delta_1(d_0f_0)=\lambda d_0f_0\] implies that $d_0f_0$ is an eigenvector of $-\Delta_1$ with eigenvalue $\lambda$. That is, the entire spectrum of $-\Delta_0$ is replicated on the level of 1-form, with the exception of the zero eigenvalue.

Next we give a description of harmonic 1-forms $\mathcal{H}_1^m$ on graphs $G_m$. In particular we want to see how  $\mathcal{H}_1^m$ changes when $m$ varies and as $m\to\infty$.

First note that there is no nontrivial  harmonic 1-form on $G_0$, since the condition $d_1^0f_1^0=0$ says the sum of $f_1^0$ over the three edges is 0, and $\delta_1^0f_1^0=$ says $f_1^0$ takes the same value on each edge.

On $G_1$, we have nine equations, three of the form
\begin{equation}d_1^1f_1^1(e_2^1)=0,\end{equation}
and six of the form
\begin{equation}\delta_1^1f_1^1(e_0^1)=0.\end{equation}
But  there is redundancy since 
\begin{equation}\sum_{e_1^1\in E_0^1}\delta_1^1f_1^1(e_0^1)=0.\end{equation}Thus we have a 1-dimensional $\mathcal{H}_1^1$. It is generated by the  1-form $h$ shown in Figure (4.1).

\begin{figure}
\scalebox{0.5}{\includegraphics{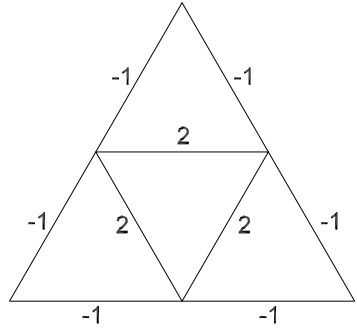}}\caption{$h$, a harmonic 1-form at level 1.}
\end{figure}
Direct computation verifies that $h_1$ is harmonic. We can also understand this 1-form as locally $d_0^1f_0^1$ where $f_0^1$ is a harmonic 0-form as in Figure (4.2):

\begin{figure}
\scalebox{0.5}{\includegraphics{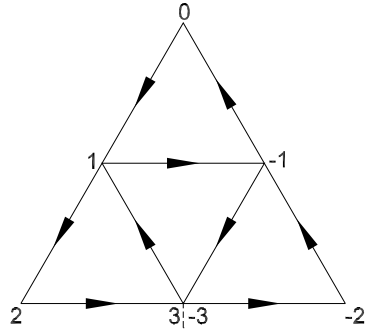}}\caption{$f_0^1$ such that $h=d_0f_0^1$.}
\end{figure}

Note that $f_0^1$ is not a well-defined function since it has ambiguous values at the bottom middle vertex. But one can consider $f_0^1$ as a harmonic mapping from $G_1$ to $\mathbb{R}/6\mathbb{Z}$ \cite{S2}. Locally $h_1=d_0^1f_0^1$ gives \[d_1^1h_1=d_1^1d_0^1f_0^1=0\]and\[\delta_1^1f_1^1=(\delta_1^1d_0^1)f_0^1=-\Delta_0^1f_0^1=0.\]

We can also use the local equality to define an extension to $G_2$, satisfying the extension property
\begin{equation*}h_1^1(e_1^1)=\sum_{e_1^2\subset e_1^1}h_1^2(e_1^2).\end{equation*} We first extend  harmonically $f_0^1$ to $f_0^2$ (Figure (4.3)) and compute $h_2=d_0^2f_0^2$ (Figure (4.4)).

\begin{figure}
\scalebox{0.5}{\includegraphics{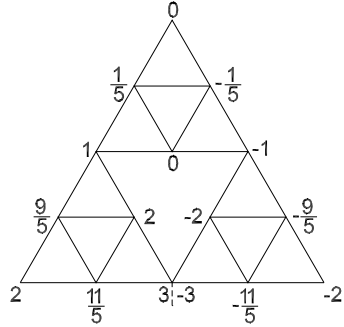}}\caption{$f_0^2$, the extension of $f_0^1$ to level 2}
\end{figure}

\begin{figure}
\scalebox{0.6}{\includegraphics{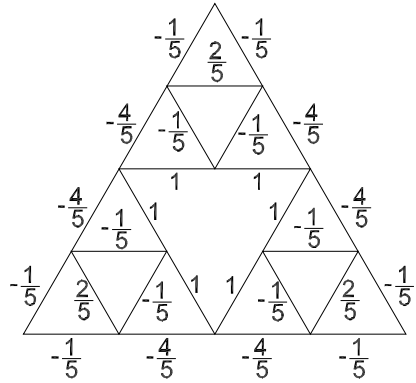}}\caption{$h_2=d_0^2f_0^2$.}
\end{figure}
More generally, we can define a local harmonic extension algorithm on any level and any cell. Suppose $f_1^m$ has values $x, y$ and $z$ on the three edges of an m-cell. Then $d_1^mf_1^m=0$ means $x+y+z=0$. When this m-cell is split into three (m+1)-cells, we extend $f_1^m$ as in Figure (4.5).

\begin{figure}
\scalebox{0.7}{\includegraphics{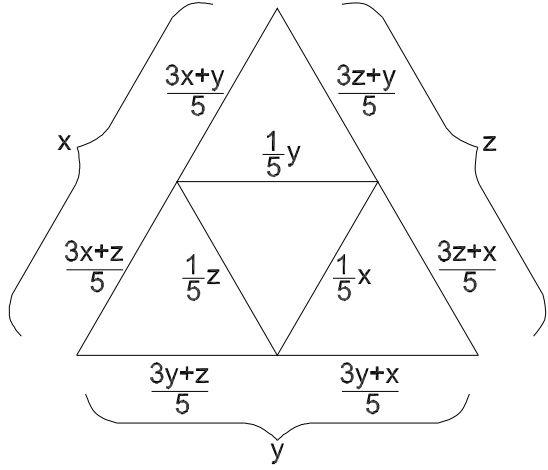}}\caption{Schematic local harmonic extension algorithm.}
\end{figure}

\begin{thm}Suppose $f_1^m$ is a harmonic 1-form on $G_m$, then $f_1^{m+1}$ derived from the local harmonic extension algorithm on each m-cell is a harmonic 1-form on $G_{m+1}$ and extends $f_1^m$ in the sense that
\begin{equation}f_1^m(e_1^{m})=\sum_{e_1^{m+1}\subset e_1^{m}}f_1^{m+1}(e_1^{m+1}).\end{equation}\end{thm} 

\begin{proof}(4.14) is clear from direct computation. The condition $d_1^{m+1}f_1^{m+1}=0$ follows from $\frac{3x+y}{5}+\frac{3z+y}{5}+\frac{y}{5}=\frac{3}{5}(x+y+z)=0$ etc. Both use the fact that $d_1^m f_1^m=0$ implies $x+y+z=0$.

To show $\delta_1^{m+1}f_1^{m+1}=0$, we need different arguments for new vertices in $G_{m+1}$ and the old vertices in $G_m$. For the new vertices one has $\frac{3x+y}{5}-\frac{3x+z}{5}+\frac{z}{5}-\frac{y}{5}=0$ etc. For the old vertices we need to use $\delta_1^m f_1^m=0$ at the same vertex and the fact that $\delta_1^{m+1}f_1^{m+1}=\frac{3}{5}\delta_1^mf_1^m=0$.
\end{proof}

We can now explicitly describe a basis for $\mathcal{H}_1^m$. There will be one basic element associated to each k-cell in $G_m$ for each $k<m$. We take the original harmonic 1-form $h$ and miniaturize to the cell and then extend harmonically. If the cell is $F_{\omega}G_{m-k}$ ($|\omega|=k$), then 
\begin{equation*}h_{\omega}(e)=\begin{cases}h\circ F_{\omega}^{-1}(e)\ (e\in F_{\omega}E_1^1)\\0 \ (elsewhere)\end{cases}\end{equation*} 
is a harmonic 1-form on $G_{k+1}$. The only nontrivial property to check is $\delta_1^{k+1}h_{\omega}=0$ on $F_{\omega}E_0^0$ but there $\delta_1^{k+1}h_{\omega}=-1-(-1)+0-0=0$. We then extend by the harmonic extension algorithm to a harmonic 1-form on $G_m$. In this way we obtain $1+3+3^2+\cdots+3^{m-1}=\frac{3^m-1}{2}$ harmonic 1-forms. This equals the dimension of $\mathcal{H}_1^m$ since
$\# E_1^m=3^{m+1}$ is the dimension of $D_1^m$, and there are $3^m$ equations for $d_1^mf_1^m=0$, $\frac{3^{m+1}+3}{2}$ equations for $\delta_1^mf_1^m=0$ with one redundancy. Therefore it suffices to check the linear independence for the harmonic 1-forms.  Actually we will show that they are orthogonal with respect to the inner product $<,>_1^m$.

\begin{thm}$<h_{\omega},h_{\omega'}>_1^m=0$ if $\omega\neq \omega'$.\end{thm}

\begin{proof}If the cells associated with $\omega$ and $\omega'$ are disjoint then the two 1-forms are trivially orthogonal since they have disjoint support. Thus we might as well assume the $\omega$ cell is contained in the $\omega'$ cell with $|\omega|=k$.

First we claim $<h_{\omega},h_{\omega'}>_1^{k+1}=0$. In this case there are only nine edges in $E_1^{k+1}$ on which $h_{\omega}$ is nonzero, the three inner edges, where $h_{\omega}=2$ and the six outer edges where $h_{\omega}=-1$. Then it suffices to show that the contributions to the inner product from each of these two types of edges is separately 0. For this it suffices to show that the sum of $h_{\omega'}$ over each type of edges vanishes. The sum over the six outer edges is exactly $\delta_1^k h_{\omega'}$, which vanishes. But the sum over the three inner edges is the difference between the sum of $\delta_1^{k+1}h_{\omega'}$ on each of the three (k+1)-cells that comprise the k-cell, and $\delta_1^k h_{\omega'}$ on the k-cell, yielding $0-0=0$.

To complete the proof it suffices to show that orthogonality is inherited by harmonic extensions. In fact we will show
\begin{equation}<f_1^{k+1},g_1^{k+1}>_1^{k+1}=\frac{3}{5}<f_1^k,g_1^k>_1^k.\end{equation} for any harmonic 1-forms $f_1^{k},g_1^{k}$ and their harmonic extensions $f_1^{k+1},g_1^{k+1}$. We compute the contribution to the inner product from each k-cell, which is
\begin{align*}(\frac{3x+y}{5})(\frac{3x'+y'}{5})+(\frac{3x+z}{5})(\frac{3x'+z'}{5})+(\frac{3y+z}{5})(\frac{3y'+z'}{5})+(\frac{3y+x}{5})(\frac{3y'+x'}{5})\\+(\frac{3z+x}{5})(\frac{3z'+x'}{5})
 +(\frac{3z+y}{5})(\frac{3z'+y'}{5})+\frac{x}{5}\frac{x'}{5}+\frac{y}{5}\frac{y'}{5}+\frac{z}{5}\frac{z'}{5}\\
 =\frac{3}{5}(xx'+yy'+zz')\end{align*}

\end{proof}

It is perhaps more natural to describe a basis of harmonic 1-forms in terms of integrals over 1-cycles that give a basis for homology up to a level $m$. If we denote by $\gamma$ the central nontrivial cycle in $\Gamma_1$ (the upside-down triangle), then the homology basis consists of 1-cycles $F_{\omega}\gamma$ for $|\omega|<m$.

\begin{lem}For any $\omega$ and $\omega'$,
\begin{equation}\int_{F_{\omega'}\gamma}h_{\omega}=\begin{cases}6\  (\omega'=\omega)\\-2\ (|\omega'|<|\omega|,F_{\omega}\mathcal{K}\cap F_{\omega'}\gamma\neq\emptyset)\\0\  (otherwise)\end{cases}\end{equation}

For fixed $\omega'$ the condition $F_{\omega}\mathcal{K}\cap F_{\omega'}\gamma\neq\emptyset$ occurs exactly when $\omega_j=\omega'_j$ for $j<m'=|\omega'|$, and $\omega_j\neq\omega_{m'+1}$ for $j>m'+1$. In particular for fixed $\omega$, the integral is zero for all $\omega'$ such that $|\omega'|>|\omega|$ and the integral is nonzero for at most 4 choices of $\omega'$.\end{lem} 

\begin{proof}From the harmonic extension algorithm it is clear that for any harmonic 1-form
\begin{equation*}\int_{F_{\omega'}\gamma}h=-\frac{1}{5}\int_{\partial F_{\omega'}\mathcal{K}}h\end{equation*} (the minus sign comes from reversed orientation).

For the original harmonic 1-form we have $\int_{\partial F_{\omega'}\mathcal{K}}h=0$ for any nonempty word $\omega'$, which is exactly the condition $d_1^{m'}h=0$, while $\int_{\partial \mathcal{K}}h=6$ by inspection. This verifies (4.16) for $\omega=\emptyset$.More generally, we can localize this argument to obtain the first and third line of (4.16) for any $\omega$.

If $|\omega'|<|\omega|$ the  we can only get a nonzero integral if $F_{\omega}\mathcal{K}\cap F_{\omega'}\gamma\neq\emptyset$, in which case the edge of $\partial F_{\omega}\mathcal{K}$ that meets $F_{\omega'}\gamma$ contributes $-2$ to the integral over $F_{\omega'}\gamma$, and elsewhere on $F_{\omega'}\gamma$, $h_{\omega}$ is zero.  Since, for fixed $\omega$, $\partial F_{\omega}\mathcal{K}$ has only three edges, there are at most three choices of $\omega'$  where we get integral $-2$ (we have fewer choices if some of the edges of $\partial F_{\omega}\mathcal{K}$ lie on $\partial\mathcal{K}$). On the other hand, if we fix $\omega'$, then $F_{\omega}\mathcal{K}$ can intersect $F_{\omega'}\gamma$ only if $F_{\omega}\mathcal{K}\subset F_{\omega'}\mathcal{K}$, so we must have $\omega_j=\omega'_{j}$ for $j\le m'$. We can then allow any choice of $\omega_{m'+1}$, after which we can never choose that symbol again.
\end{proof}

Using (4.16) it is easy to set up an inductive procedure to produce a new basis $\widetilde{h}_{\omega}$ of harmonic 1-forms that satisfies
\begin{equation}\int_{F_{\omega'}\gamma}\widetilde{h}_{\omega}=\delta_{\omega,\omega'}.\end{equation}
Then\begin{equation}\widetilde{h}=\sum c_{\omega}\widetilde{h}_{\omega}\end{equation} produces a harmonic 1-form with
\begin{equation}\int_{F_{\omega'}\gamma}\widetilde{h}=c_{\omega'}.\end{equation}
for any values $c_{\omega'}$, only finite of which are nonzero. This implements the homology/ cohomology duality between harmonic 1-forms and 1-chains at finite levels.

Next we consider passing to the limit to obtain 1-forms on SG. It is clear that the harmonic extension algorithm may be used infinitely often to obtain the values on all edges in $\cup_{m=0}^{\infty}E_1^m$ for 1-forms. We claim then on each edge we have a finite measure. Since each of the harmonic 1-forms is locally $d_0f_0$ for some harmonic function $f_0$, this is equivalent to the statement that the restriction of a harmonic function to an edge is of bounded variation. But in \cite{DVS} it is shown that the restriction of a harmonic function to an edge is either monotonic, or has a single maximum or minimum, and monotonic functions are of bounded variation. We also observed that $<h_{\omega},h_{\omega}>_1$ is finite.

However, if we only assume $f_0$ has finite energy then we cannot conclude $d_0f_0$ restricted to an edge is a measure. In other words, the restriction of $f_0$ to an edge may not be of bounded variation. For example, consider $f_0^m$ that vanishes everywhere except on a boundary edge, where it oscillates between $\pm 2^{-m/2}$. Then $\mathcal{E}_0^m(f_0^m)\le 4\cdot 2^m(2^{-m/2})^{2}=4$ and harmonic extension of $f_0$ will have energy no larger than 4. But the variation of $f_0^m$ along the boundary edge is at least $2\cdot 2^{-m/2}\cdot 2^m=2\cdot 2^{m/2}$. So there is no estimate of total variation in terms of energy. It is then a routine matter to produce a counterexample by taking an appropriate infinite series of localization of these functions.

\begin{thm}Assume $f_0\in dom(\Delta)$, then the restriction of $f_0$ to any edge is of bounded variation, thus $d_0f_0$ is a 1-form (its restriction to every edge is a measure).\end{thm}

\begin{proof}We have already observed that the statement is true of harmonic functions. By localization it suffices to prove this for one of the boundary edges in $E_1^0$. By subtracting a harmonic function we may reduce this to the case where $f_0$ vanishes at the boundary vertices. Then we may write
\begin{equation}f_0(x)=\int_{SG}G(x,y)(-\Delta f_0(y))dy\end{equation}
where $G$ is the Green's function. Thus it suffices to show the restriction of $G(\cdot,y)$ to the boundary edges for each $y$ is of bounded variation with total variation bounded by a constant  independent of $y$. This follows from Kigami's formula:
\begin{equation}G(x,y)=\sum_{m=0}^{\infty}\sum_{|\omega|=m}(\frac{3}{5})^m\Psi(F_{\omega}^{-1}x, F_{\omega}^{-1}y),\end{equation}
where $\Psi$ is an explicit function that is piecewise harmonic at level 1.

The key observation is that for each fixed $y$ there is at most one $\omega$ with $|\omega|=m$ for which $ \Psi(F_{\omega}^{-1}x, F_{\omega}^{-1}y)$ is not identically zero for all $x$ along a fixed boundary edge. Since $\Psi$ is piecewise harmonic, its restrictions are of bounded variation and the variation does not change when we localize. Thus we get an estimate of $c\sum (\frac{3}{5})^m$ for the variation of the restrictions, which is independent of $y$.
\end{proof}

In fact the argument works if we only assume $\Delta f_0$ exists in $L^1$ or $L^2$. A more challenging problem is to find the largest space between $dom(\mathcal{E})$ and $dom(\Delta)$ that yields the same conclusion. 

In Section 5 we will show that the restrictions of nonconstant harmonic functions are not absolutely continuous, thus the measures $d_0f_0$ on edges are singular with respect to the arc length measure. 

Another important question is under what conditions we can allow an infinite series in (4.18). It is simpler to consider the series 
\begin{equation}
h=\sum c_{\omega}h_{\omega},
\end{equation} 
since the condition
\begin{equation}
\sum(\frac{5}{3})^{|\omega|}|c_{\omega}|^2<\infty
\end{equation}
is natural if we want $<h,h>_1<\infty.$ But again we an show that (4.23) is not sufficient to make (4.22) converge on edges. To see this, consider a finite sum with $c_{\omega}=(\frac{3}{10})^{m/2}$ for $|\omega|=m$ with all $\omega_j=1$ or $2$ and $c_{\omega}=0$ otherwise. Since there are $2^m$ nonzero values we see (4.23) is $2^m\cdot(\frac{5}{3})^m\cdot(\frac{3}{10})^m=1$. However, if we choose $e_1$ to be the bottom boundary edge then $h(e_1)=2^m\cdot (\frac{3}{10})^{m/2}(-2)=-2(\frac{6}{5})^{m/2}$.

Next we discuss a theory of 2-forms on SG and the mappings $d_1$ from 1-forms to 2-forms and $\delta_2$ from 2-forms to 1-forms. The space of 2-forms will just be the space of finite signed measures on SG, but of course this is too large a space, and we want to identify a space of `smooth' measures for which $\delta_2$ may be defined as a `trace on line segments'. Let $\mu$ denote the standard balanced measure on SG, so $\mu$ assigns $1/3^m$ to each m-cell. If $f$ is a continuous function  on SG then $fd_\mu$ is a measure absolutely continuous with respect to $\mu$ with continuous Radon-Nykodim derivative. We could allow a somewhat larger class of Radon-Nykodim derivatives, but for now this will suffice. If $\sigma_1^m$ is any edge of level m, then 
\begin{equation}
\delta_2(fd\mu)(\sigma_1^m)=\int_{\sigma_1^m}f
\end{equation} 
is well-defined, and in fact\begin{equation}
\delta_2(fd\mu)(\sigma_1^m)=\lim_{n\to\infty}(\frac{3}{2})^n\sum_{\sigma_2^n\cap\sigma_1^m\neq\emptyset}\int_{\sigma_2^n}fd\mu,
\end{equation} 
which is a renormalized limit of the measure of a `thickening' $\cup_{\sigma_2^n\cap\sigma_1^m\neq\emptyset}\sigma_2^n$ of $\sigma_1^m$. 

Note that with this definition there are no nonzero harmonic 2-forms.

The corresponding definition for $d_1$ is
\begin{equation}
d_1f_1(\sigma_2^m)=\lim_{n\to\infty}(\frac{2}{3})^n\sum_{\sigma_1^n\subset\sigma_2^m}f_1(\sigma_1^n),
\end{equation} 
but the issue is what the space of 1-forms should be admissible in this definition. Note that if $f_1=d_0f_0$ then the sums on the right-hand side of (4.26) are always zero, so the limit exists (regardless of the renormalizing factor), and we have $d_1d_0=0$ and  $d_1h=0$ for harmonic 1-forms. On the other hand, if we assume that $f_1=\delta_2(fd\mu)$ for some continuous function $f$, then $d_1f_1$ exists and 
\begin{equation}
d_1\delta_2(fd\mu)=3fd\mu.
\end{equation} 
The factor 3 arises because each cell $\sigma_2^m$ has 3 edges $\sigma_1^m$ contained in it. So this yields the relatively trivial result that $\Delta_2=3I$, and $\Delta_1$ restricted to this class of 1-forms is also $3I$.

We conclude this section by introducing an alternative approach. Choose a finite measure $\nu$ on SG. In place of the inner product (4.2) on 0-forms we consider
\begin{equation}
<f_0,g_0>_0'=\int_{SG}f_0g_0d\nu.
\end{equation} 
This does not change the definition of $d_0$, but it will change the definition of $\delta_1$. In place of the weight $\mu_0^m(e_0^m)$ defined in (4.1) we use
\begin{equation}
\mu_0'(e_0^m)=\int_{SG}\Psi_{e_0^m}d\nu
\end{equation} 
where $\Psi_{e_0^m}$ denotes the piecewise harmonic spline on level m that assumes the value 1 at $e_0^m$ and the value 0 at all other $E_0^m$ vertices. Then in place of (4.9) we have
\begin{equation}
\delta_1'^{m}f_1^m(e_0^m)=\frac{(5/3)^m}{\int\Psi_{e_0^m}d\nu}\sum_{e_1^m\in E_1^m}sgn(e_0^m,e_1^m)f_1^m(e_1^m).
\end{equation} 
Thus $\delta_1'=\lim_{m\to\infty}\delta_1'^m$ when the limit exists, and $\Delta_0=\delta_1' d_0$. The point is that we can identify $\Delta_0'$ with Kigami's definition (\cite{Ki},\cite{SB}) of $\Delta_\nu$, namely,
\begin{equation}
\Delta_{\nu}f_0(e_0^n)=\lim_{m\to\infty}\delta_1'^md_0^mf_0(e_0^m) \ (uniform\ limit).
\end{equation} 
This does not change the definition of harmonic 1-forms, as they will also have $\delta_1'^mf_1=0$ for large enough m.

We would next want to consider the subspace of 2-forms of the form $fd\nu$ for continuous functions $f$. We would then need to replace the renormalization coefficients in (4.25) and (4.26) to yield the appropriate notion of `trace on the line segments' for the measures $fd\nu$. This of course depends on the measure $\nu$, and we do not know how to deal with the general case. What we can do is to describe the case of the Kusuoka measure. For any harmonic function $h$, define the energy measure $\nu_h$ by $\nu_h(C)=$ the energy of $h$ restricted to $C$ (it is obvious how to do the restriction when $C$ is a union of cells and this suffices for the definition of the measure). The Kusuoka measure $\nu=\nu_{h}+\nu_{h'}$, where $h$ and $h'$ is an orthonormal basis of harmonic functions modulo constants in the energy inner product (it is easy to see that the result is independent of the choice of orthonormal basis).

For simplicity we look at the case where the 2-form is $\nu_h$ for some harmonic function $h$, and $\sigma_1$ is the line segment joining $q_1$ and $q_2$. Then we need the growth rate of $\nu_h(\Omega_n)=\sum\nu_h(\sigma_2^n)$, where $\Omega_n=\cup F_{\omega}(SG)$, the union is over all words $\omega$ of length $n$ with $\omega_j=1$ or $2$, and the sum is over all $\sigma_2^n\cap\sigma_1\neq\emptyset$.

Now it happens that the value of $\nu(\Omega_n)$ was computed exactly in [OS] (based on results in [AHS]) to be of the form
\begin{equation}
a(\frac{17+\sqrt{73}}{30})^n+b(\frac{17-\sqrt{73}}{30})^n
\end{equation} 
where the constants $a$ and $b$ are explicitly determined by the values of $h$ on $E_0^1$.This means we want to define
\begin{equation}
\delta'_2(fd\nu)(\sigma_1^m)=\lim_{n\to\infty}(\frac{30}{17+\sqrt{73}})^n\sum_{\sigma_2^n\cap\sigma_1^m\neq\emptyset}\int_{\sigma_2^n}fd\nu.
\end{equation} 
It is easy to localize the above argument to show that the limit exists if $f$ is continuous. What is not clear  is how to characterize the space of 1-forms that we obtain, and how to define $d_1'$ on this space of 1-forms.

\section{Singularity of 1-form Measures}
In this section we examine more closely the nature of the measures on the line segments given by 1-forms of the type studied in Section 4. The main result is that $d_0f_0$ on a line segment for $f_0$ a nonconstant harmonic function is not absolutely continuous with respect to the Lebesgue measure. The same result follows immediately for harmonic 1-forms, since they are locally equal to $d_0f_0$. It should be straightforward to extend the result to $d_0f_0$ for $f_0\in dom(\Delta_0)$ since such functions may be well approximated by harmonic functions \cite{S}. A more difficult question we cannot answer is whether or not the measure is completely singular with respect to Lebesgue measure (has zero absolute continuous part).
\begin{thm}
Let $f_0$ be a nonconstant harmonic function, and $L$ any line segment. Then $d_0f_0|_{L}$ is not absolutely continuous with respect to Lebesgue measure on $L$.  
\end{thm} 

\begin{proof}
Let $x(t)$ denote the standard parametrization of $L$ for $0\le t\le1$. Then $f_0(x(t))$ satisfies the harmonic extension algorithm:
\begin{equation}
\begin{cases}&f_0(x(\frac{2j+1}{2^{m+1}}))=\frac{8}{25}f_0(x(\frac{j}{2^m}))+\frac{4}{5}f_0(x(\frac{j+1}{2^m}))-\frac{3}{25}f_0(x(\frac{j+2}{2^m}))\\
&f_0(x(\frac{2j+3}{2^{m+1}}))=\frac{8}{25}f_0(x(\frac{j+2}{2^m}))+\frac{4}{5}f_0(x(\frac{j+1}{2^m}))-\frac{3}{25}f_0(x(\frac{j}{2^m}))\end{cases}\end{equation}
for $j$ even (\cite{DVS} Algorithm 2.2). This allows us to pass from information at $t=\frac{j}{2^m}$ for all $j$ to information at $t=\frac{j}{2^{m+1}}$ for all $j$. Now if $\nu$ is a nonatomic measure on $[0,1]$, the dyadic approximations
\begin{equation}
g_m=\sum_{j=0}^{2^m-1}2^m\nu(I_j^m)\chi_{I_j^m}
\end{equation}  
where $I_j^m=[\frac{j}{2^m},\frac{j+1}{2^m}]$ determine whether or not $\nu$ is absolutely continuous, namely $\nu=gdt$ if and only if $g_m\to g$ in $L^1$. For our measure
\begin{equation}
\nu(I_j^m)=f_0(x(\frac{j+1}{2^m}))-f_0(x(\frac{j}{2^m})),
\end{equation} 
so we may use (5.1) to obtain a lower bound for $\|g_{m+1}-g_{m}\|_1$. We may write
\[g_m=\sum_{j=0}^{2^{m}-1}2^{m+1}\frac{\nu(I_{2j}^{m+1})+\nu(I_{2j+1}^{m+1})}{2}(\chi_{I_{2j}^{m+1}}+\chi_{I_{2j+1}^{m+1}}),\] so
\begin{equation}
\|g_{m+1}-g_{m}\|_1=\sum_{j=0}^{2^m-1}|\nu(I_{2j}^{m+1})-\nu(I_{2j+1}^{m+1})|.
\end{equation} 
So consider the contribution of the two consecutive terms in (5.4). For $j$ even we have
\begin{align*}|\nu(I_{2j}^{m+1})&-\nu(I_{2j+1}^{m+1})|+|\nu(I_{2j+2}^{m+1})-\nu(I_{2j+3}^{m+1})|\\
&=|2f_0(x(\frac{2j+1}{2^{m+1}}))-f_0(x(\frac{j}{2^m}))-f_0(x(\frac{j+1}{2^m}))|\\&+|2f_0(x(\frac{2j+3}{2^{m+1}}))-f_0(x(\frac{j+1}{2^m}))-f_0(x(\frac{j+2}{2^m}))|\\
&=|\frac{3}{5}f_0(x(\frac{j+1}{2^{m}}))-\frac{9}{25}f_0(x(\frac{j}{2^m}))-\frac{6}{25}f_0(x(\frac{j+2}{2^m}))|\\&+|\frac{3}{5}f_0(x(\frac{j+1}{2^{m}}))-\frac{6}{25}f_0(x(\frac{j}{2^m}))-\frac{9}{25}f_0(x(\frac{j+2}{2^m}))|
\end{align*}
by (5.3) and (5.4). If we write $a=f_0(x(\frac{j}{2^m}))$, $b=f_0(x(\frac{j+1}{2^m}))$ and $c=f_0(x(\frac{j+2}{2^m}))$, this is \begin{equation}
|\frac{3}{5}b-\frac{9}{25}a-\frac{6}{25}c|+|\frac{3}{5}b-\frac{6}{25}a-\frac{9}{25}c|.
\end{equation} 
If we fix $a$ and $c$ and vary $b$, then a lower bound for (5.5) is $\frac{3}{25}|c-a|$ by the triangle inequality. Substituting this lower bound into (5.4) yields
\begin{align*}\|g_{m+1}-g_{m}\|_1&\ge\frac{3}{25}\sum_{j=0}^{2^{m-1}+1}|f_0(x(\frac{j+1}{2^{m-1}}))-f_0(x(\frac{j}{2^{m-1}}))|\\&\ge\frac{3}{25}|f_0(x(1))-f_0(x(0))|,\end{align*} 
which shows $\{g_m\}$ does not converge in $L^1$ (if $f_0(x(1)))=f_0(x(0))$, just pass to a subinterval).

\end{proof}

\section{ 3-dimensional Sierpinski Gasket}

The 3-dimensional Sierpinski gasket may be realized in $\mathbb{R}^3$ as the attractor of the IFS
\begin{equation*}F_k x=\frac{1}{2}(x-q_k)+q_k, \end{equation*} where $E_0^0:=\{q_k\}_{k=0}^3$ are vertices of a regular tetrahedron. $E_1^0$ is defined to be the collection of edges $[q_i,q_j] (i\neq j)$, $E_2^0$ the collection of 2-dimensional faces $[q_i,q_j,q_k]$ ($i,j,k$ are distinct) and $E_3^0$ the simplex $[q_0,q_1,q_2,q_3]$.

\subsection{Graph approximation of 3-dimensional Sierpinski gasket}

For $m\in\mathbb{N}$, define $E_j^{m}$ inductively by \begin{equation*}E_j^{m}=\bigcup_k F_k E_j^{m-1}.\end{equation*}For $j=1,2,3$, these are disjoint unions. For $j=0$, we identify $F_j q_k=F_k q_j$ for $k\neq j$ and  it follows that $E_0^{0}\subset E_0^{1}\subset E_0^{2}\subset\cdots.$ It is easy to see that \begin{equation}
\begin{cases}
\# E_0^{m}&=2\cdot4^m+2\\
\# E_1^{m}&=6\cdot4^m\\
\# E_2^{m}&=4^{m+1}\\
\# E_3^{m}&=4^m.
\end{cases}
\end{equation}

Since all intersections occur at vertices, only the theory of 0-forms and 1-forms is nontrivial. In this section we discuss the harmonic 1-forms at each level.

We choose the orientation on 2-dimensional faces that views them from the `outside', that is, we take $[q_0,q_1,q_2]$, $[q_1,q_0,q_3]$, $[q_2,q_3,q_0]$ and $[q_3,q_2,q_1]$ as positive orientations on $E_2^{0}$ and carry them over to $E_2^{m}$. 

On the other hand, there is no consistent orientation on the edges. The counterclockwise convention has the boundary edges of $[q_0,q_1,q_2]$ with positive orientation being $[q_0,q_1]$, $[q_1,q_2]$ and $[q_2,q_0]$.  However, $[q_1,q_0]$ is of positive orientation on $[q_1,q_0,q_3]$. Similar in  consistency exists for every edge and can be written as
\begin{equation*}\sum_{e_2^0\supset e_1^0}sgn(e_1^0,e_2^0)=0.\end{equation*}
Similarly on every level one has
\begin{equation}\sum_{e_2^m\supset e_1^m}sgn(e_1^m,e_2^m)=0.\end{equation}

\subsection{Equations for harmonic 1-forms}
In analogy with the Sierpinski gasket, we have 
\begin{equation}d_1^{m}f_1^{m}(e_2^{m})=\sum_{e_1^{m}}sgn(e_1^{m},e_2^{m})f_1^{m}(e_1^{m})\end{equation}
and
\begin{equation}\delta_1^{m}f_1^{m}(e_0^{m})=\sum_{e_1^{m}}sgn(e_0^{m},e_1^{m})f_1^{m}(e_1^{m}),\end{equation}
where the scaling factors are left out for simplicity.

The equations for harmonic 1-forms $h_1^{m}\in\mathcal{H}^{m}_1$ are
\begin{equation}d_1^{m}h_1^{m}(e_2^{m})=0\end{equation}
and 
\begin{equation}\delta_1^{m}h_1^{m}(e_0^{m})=0\end{equation}
for each $e_2^{m}\in E^{m}_2$ and $e_0^{m}\in E^{m}_0$.

There are redundancies in these equations. For each simplex $e_3^{m}$, one has
\begin{equation}\sum_{e_2^{m}\subset e_3^{m}}d_1^{m}h_1^{m}(e_2^{m})=0\end{equation}
because the left-hand side equals 
\[\sum_{e_1^{m}\subset e_3^{m}}h_1^{m}(e_1^m)\sum_{e_2^m\supset e_1^m}sgn(e_1^m,e_2^m)\]where the inner sum vanishes by (6.2). Thus there are at most $3\cdot 4^m$ independent relations in (6.5).

Similarly, we have 
\begin{equation}\sum_{e_0^m\in E_0^{m}}\delta_1^{m}h_1^{m}(e_0^m)=0\end{equation}
since the sum can be written as 
\begin{equation*}\sum_{e_1^m\in E_1^{m}}h_1^{m}(e_1^m)\sum_{e_0^m\subset E_0^{m}}sgn(e_0^m,e_1^m)\end{equation*}and the inner sum vanishes because each edge has two vertices with opposite signs. Thus (6.6) gives at most $2\cdot 4^m+1$ independent relations.

Now $\# E_1^{m}-3\cdot4^m-(2\cdot4^m+1)=4^m-1$ suggests $dim(\mathcal{H}_1^{m})=4^m-1$. We demonstrate this by constructing an orthogonal basis for $\mathcal{H}_1^{m}$ containing $4^m-1$ elements.

Since $4^0-1=0$, there is no harmonic 1-form on level 0. However, we construct a 1-form in Figure 6.1 satisfying $d_1^{0}f_1^{0}(e_2^0)=0$ on each $e_2^{0}\in E_2^{0}$ and $\delta_1^{0}f_1^{0}(e_0^0)=0$ at $q_0$ and $q_1$ with $\delta_1^{0}f_1^{0}(q_2)=-\delta_1^{0}f_1^{0}(q_3)=2$. Note that by rotating this example we can make $\delta_1^{0}f_1^{0}(e_0^0)$ vanish at any pair of vertices. 

\begin{figure}
\scalebox{0.5}{\includegraphics{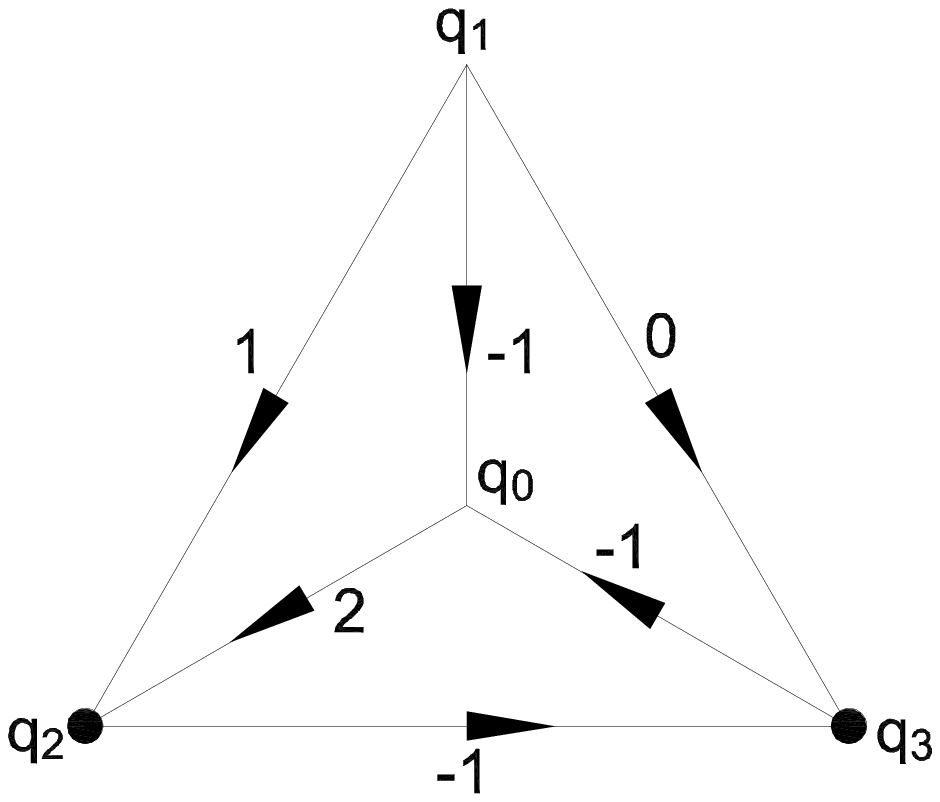}}

\caption{}

\end{figure}


By placing appropriate rotations of the above example in 3 of the 4 simplices and 0 on the fourth, we construct harmonic 1-forms on level 1. We just have to make sure that the  outer vertices are zeroes of $\delta_1^1 h_1^1$, as are the three vertices joining the simplex where $h_1^1$ is zero.  Along the triangle where the three simplices intersect, alternate $+2$ and $-2$ for $d_1^1 h_1^1$.

\begin{figure}
\scalebox{0.5}{\includegraphics{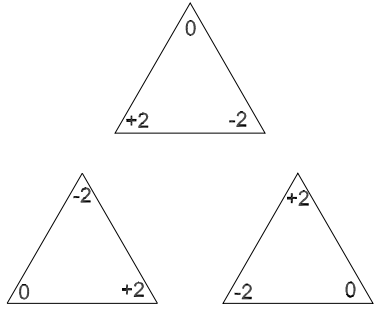}}
\caption{Placement of the example from Figure 6.1 in the three bottom simplices with values of $\delta_1^1$ noted at the bottom vertices. Values in the top simplex are 0.}
\end{figure}
This gives us a basic harmonic 1-forms in $\mathcal{H}_1^1$. By rotating we get 4 such 1-forms. But they are not linearly independent as they sum to 0. Choose an orthogonal basis, say, $A_1, A_2, A_3$ for the span.

\subsection{Integration along homology cycles}Now for each face $e_2^0$ there is a cycle $\gamma$ consisting of the inner triangle with three edges in $E_1^1$. 
For each $h_1^1\in\mathcal{H}_1^1$ we are interested in the integral $\int_{\gamma}h_1^1$ as this gives us a cohomology/homology pairing.

In view of (6.5), we have
\begin{equation}\int_{\gamma}h_1^1=\int_{\partial e_2^0}h_1^1,\end{equation}and 
\begin{equation}\sum_{j=1}^4 \int_{\gamma_{j}}h^1_1=0.\end{equation}

A direct computation shows $\int_{\gamma_j}h_1^1$ takes the values $(3,-1,-1,-1)$ for the 1-form in Figure 6.2 with the value 3 on the `bottom' face. Any three of these are linear independent, and any values of  $\int_{\gamma_j}h_1^1$ satisfying (6.10) can be attained by a harmonic 1-form in $\mathcal{H}_1^1$.

\subsection{Harmonic extension algorithm}
In order to extend the harmonic 1-forms in $\mathcal{H}_1^m$ to $\mathcal{H}_1^{m+1}$, it is convenient to observe that harmonic 1-forms may be written as $d_0f_0$ where $f_0$ is a harmonic mapping to $\mathbb{R}/\mathbb{Z}$. For example the 1-form in Figure 6.1 is  $d_0 f_0$ for the 0-form in Figure 6.3.

\begin{figure}
\scalebox{0.5}{\includegraphics{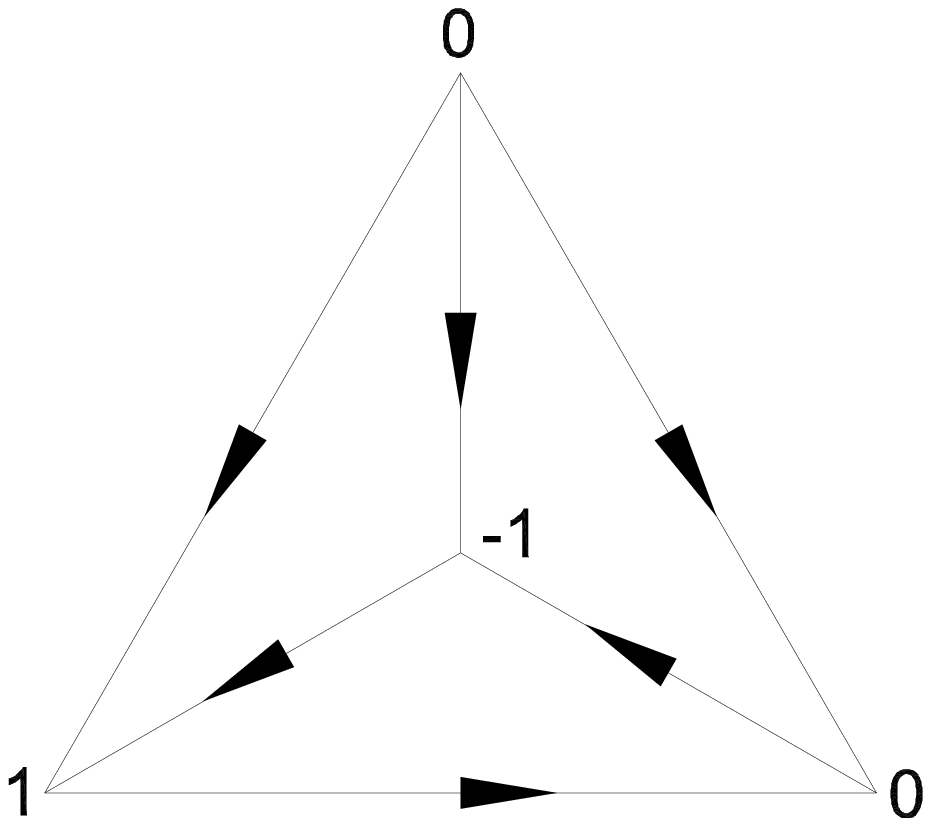}}
\caption{}
\end{figure}

When gluing together rotated copies of this 0-form we obtain contradictory values at junction vertices, but the values are the same in $\mathbb{R}/\mathbb{Z}$. It is easy to verify $\delta_1^1 d_1^1 f_0^1=0$ at all vertices.

To extend $f_0^1$ to $f_0^2$ we use the $\frac{1}{6}-\frac{1}{3}$ rule for extending harmonic functions since this is a local formula for each $e_3^1$. More generally, if $p_0,p_1,p_2$ and $p_3$ are vertices of any simplex with value $h(p_0),h(p_1),h(p_2)$ and $h(p_3)$ given, define
\begin{equation}h(p_{01})=\frac{1}{3}(h(p_0)+h(p_1))+\frac{1}{6}(h(p_2)+h(p_3))\end{equation} where $p_{01}$ denotes the midpoint of $\{p_0,p_1\}$, with a similar formula for other points.

A direct computation shows that the equations for harmonic functions are satisfied at all midpoints $p_{jk}$. At the original point $p_0$, one has 
\begin{equation*}3h(p_{0})-h(p_{01})-h(p_{02})-h(p_{03})=\frac{2}{3}(3h(p_0)-h(p_1)-h(p_2)-h(p_3)). \end{equation*}Similar equations hold at all original points, so that any harmonic condition at the original points  will be inherited by the extended functions: if $f_0^m$ is a harmonic mapping then the extension $f_0^{m+1}$ is also harmonic on the next level.

For $h_1^m\in\mathcal{H}_1^m$ we have $h_1^m=d_0^m f_0^m$ for some harmonic mapping, that is,
\begin{equation}h_1^m([p_j,p_k])=f_0^m(p_k)-f_0^m(p_j).\end{equation}

(6.11) and (6.12) give
\begin{equation}h_1^{m+1}([p_{0},p_{01}])=\frac{1}{3}h^m_1([p_0,p_1])+\frac{1}{6}h^m_1([p_0,p_2])+\frac{1}{6}h^m_1([p_0,p_3]) \end{equation}and
\begin{equation}h_1^{m+1}([p_0,p_{12}])=\frac{1}{6}h_1^m([p_1,p_2]).\end{equation}
We may simply take (6.12) and (6.13) (and similar formulas for other $e_1^{m+1}$) as our extension algorithm.

A direct calculation shows $h_1^{m+1}([p_0,p_{01}])+ h_1^{m+1}([p_{12},p_{0}])+h_1^{m+1}([p_{01},p_{12}])=0$, etc, follows from (6.13) and (6.14). That is, the analog of (6.5) on (m+1)-level holds.
To show that we have an extension, namely, 
\begin{equation}h_1^{m+1}([p_0,p_{01}])+h_1^{m+1}([p_{01},p_{1}])=h_1^{m}([p_0,p_1]),\end{equation}
we simply apply  (6.13), (6.14) and (6.5) to the faces $[p_0,p_1,p_2]$ and $[p_0,p_1,p_3]$.
The conditions $\delta_1^{m+1}h_1^{m+1}(p_{01})=0$ etc at midpoints follow from (6.13) and (6.14).
Note that
\begin{align}h_1^{m+1}([p_0,p_{01}])+h_1^{m+1}([p_0,p_{02}])+h_1^{m+1}([p_0,p_{03}])&=\\ \notag\frac{2}{3}(h_1^{m+1}([p_0,p_{1}])+&h_1^{m+1}([p_0,p_{2}])+h_1^{m+1}([p_0,p_{3}]))\end{align}
and similar equations follow directly from (6.13). Then (6.16) allows us to pass from $\delta_1^m h_1^m(p_0)=0$ to $\delta_1^{m+1}h_1^{m+1}(p_0)=0$. Thus the analog of (6.6) for (m+1)-level is also true for all $e_0^{m+1}\in E_0^{m+1}$.

\begin{thm}The dimension of $\mathcal{H}_1^m$ is $4^m-1$. There is a unique extension mapping from $\mathcal{H}_1^{m}$ to $\mathcal{H}_1^{m+1}$ given by (6.13) and (6.14), where `extension' means (6.15).

Moreover, the extension satisfies \begin{equation}\int_{\gamma}h_1^{m+1}=0\end{equation}
if $\gamma$ is a homology cycle in some $e_2^{m}\in E_2^{m}$.
\end{thm}

\begin{proof}Condition (6.17) holds because the integral over $\gamma$ is the difference between the integral over $\partial e_2^{m}$ and the sum of integrals over the boundaries of the three faces $e_2^{m+1}$ contained in $e_2^m$, all of which vanish because (6.5) holds for $m$ and $m+1$.

Denote by $A_1,A_2,A_3$ the basis of $\mathcal{H}^1_1$ extended to 1-forms on the m-level. A basis for $\mathcal{H}_1^m$ is $A_j \circ F_{\omega}^{-1}$, where $j=1,2,3$ and $|\omega|\le m-1$, which contains $3\cdot (1+4+4^2+\cdots+4^{m-1})=4^m-1$ elements (It is easy to check the linear independence by (6.17)).

If we choose a set of 3 homology cycles in the faces of each simplex $e_3^{k}$ for $k\le m-1$, then we get $4^m-1$ independent cycles and we can $uniquely$ specify the integrals over each cycle for functions in the span of our basis.
\end{proof}
Analogous to Theorem 4.2, we can show our basis is orthogonal.

(S. Aaron) Reed College, Oregon, USA\\ \textit{E-mail address:} skye.aaron@gmail.com

(Z. Conn) Rice University, Texas, USA\\ \textit{E-mail address:} zpc1@rice.edu

(R. Strichartz) Department of Mathematics, Cornell University, NY, USA\\ \textit{E-mail address:} str@math.cornell.edu

(H. Yu) The Chinese University of Hong Kong, Hong Kong\\ \textit{E-mail address:} huiyu0606@gmail.com
\end{document}